\documentclass[12pt]{amsart}

\usepackage{amsfonts}
\usepackage{amssymb}
\usepackage{amsthm}

\usepackage{color}

\newtheorem{thm}{Theorem}[section]
\newtheorem{lemma}[thm]{Lemma}

\theoremstyle{definition}
\newtheorem*{thmA}{Theorem A}
\newtheorem*{thmB}{Theorem B}
\newtheorem*{thmC}{Theorem C}
\theoremstyle{definition}

\numberwithin{equation}{section}

\newcommand{\N}{\mathbb N}

\begin{document}

\author[G.A.\ Fern\'andez-Alcober]{Gustavo A. Fern\'andez-Alcober}
\address{Department of Mathematics, University of the Basque Country UPV/EHU,
48080 Bilbao, Spain}
\email{gustavo.fernandez@ehu.es}

\author[M.\ Morigi]{Marta Morigi}
\address{Dipartimento di Matematica, Universit\`a di Bologna,
Piazza di Porta San Donato 5, 40126 Bologna, Italy}
\email{marta.morigi@unibo.it}

\author[P.\ Shumyatsky]{Pavel Shumyatsky}
\address{Department of Mathematics, University of Brasilia,
Brasilia-DF, 70910-900 Brazil}
\email{pavel@unb.br}

\title[Coverings of commutators]{Procyclic coverings of commutators in profinite groups}

\subjclass[2010]{20E18; 20F14.}
\keywords{Profinite groups; procyclic subgroups; commutators}
\thanks{The first and second authors are supported by the Spanish Government, grant
MTM2011-28229-C02-02.
The first and third authors are supported by the Brazilian and Spanish Governments, under the project with the following references: Capes/DGU 304/13; PHB2012-0217-PC.
The first author is also supported by the Basque Government, grant IT753-13.
The second author is also supported by INDAM (GNSAGA)}

\begin{abstract}  We consider profinite groups in which all commutators are contained in a union of finitely many procyclic subgroups. It is shown that if $G$ is a profinite group in which all commutators are covered by $m$ procyclic subgroups, then $G$ possesses a finite characteristic subgroup $M$ contained in $G'$ such that the order of $M$ is $m$-bounded and $G'/M$ is procyclic. If $G$ is a pro-$p$ group such that all commutators in $G$ are covered by $m$ procyclic subgroups, then $G'$ is either finite of $m$-bounded order or procyclic.

\end{abstract}

\maketitle

\section{Introduction}
A covering of a group $G$ is a family $\{S_i\}_{i\in I}$ of subsets of
$G$ such that $G=\bigcup_{i\in I}\,S_i$.
If $\{H_i\}_{i\in I}$ is a covering of $G$ by subgroups, it is natural to ask
what information about $G$ can be deduced from properties of the subgroups $H_i$. In the case where the covering is finite actually quite a lot about the structure of $G$ can be said. In particular, as was first pointed out by Baer (see \cite[p.\ 105]{Rob}), a group covered by finitely many cyclic subgroups is either cyclic or finite. Fern\'andez-Alcober and Shumyatsky proved that if $G$ is a group in which the set of all commutators is covered by finitely many cyclic subgroups, then the derived group $G'$ is either finite or cyclic \cite{FerShu}. Later, in \cite{CN}, Cutolo and Nicotera showed that if $G$ is a group in which the set of all $\gamma_{j}$-commutators is covered by finitely many cyclic subgroups, then $\gamma_{j}(G)$ is finite-by-cyclic. They also showed that $\gamma_{j}(G)$ can be infinite and not cyclic. It is still unknown whether  a similar result holds for the derived words $\delta_{j}$.

Recall that a profinite group is a topological group that is isomorphic
to an inverse limit of finite groups.  The textbooks \cite{ribes-zal}  and \cite{book:wilson} provide a good introduction to the theory of profinite groups. In the context of profinite groups all the usual concepts of group theory are interpreted topologically. In particular, the derived group $G'$ of a profinite group $G$ is the {\it closed} subgroup generated by all commutators in $G$.

In this paper we examine profinite groups in which all commutators are covered by finitely many procyclic subgroups. Our natural expectation was 
that the derived subgroup in such a group should either be finite or procyclic, but this turned out to be false. Indeed, let $A$ be a finite group such that $A'$ is noncyclic of order four, and let $B$ be a pro-$p$ group such that $B'$ is infinite procyclic, where $p$ is an odd prime. Then it is easy to see that $G=A\times B$ is a profinite group in which $G'$ is infinite, not procyclic, and can be covered by 3 procyclic subgroups.

However, we can prove the following result.

\begin{thmA} Let $m$ be a positive integer and let $G$ a profinite group in which all commutators are covered by $m$ procyclic subgroups. Then $G$ possesses a finite characteristic subgroup $M$ contained in $G'$ such that the order of $M$ is $m$-bounded and $G'/M$ is procyclic.
\end{thmA}

As usual, we use the expression ``$a$-bounded" to mean ``bounded from above by some function depending only on the parameter $a$".

Further, we concentrate on pro-$p$ groups in which all commutators are covered by finitely many procyclic subgroups. In this case our initial expectation that $G'$ is either finite or procyclic has been confirmed.

\begin{thmB} Let $p$ be a prime and let $G$ be a pro-$p$ group such that all commutators in $G$ are covered by $m$ procyclic subgroups. Then $G'$ is either finite of $m$-bounded order or procyclic.
\end{thmB}

The above results are not the first that deal with coverings of word-values in profinite groups. For a family of group words $w$ it was shown in \cite{mz13} that if $G$ is a profinite group in which all $w$-values are contained in a union of finitely many closed subgroups with a prescribed property,
then the verbal subgroup $w(G)$ has the same property as well. More recently the results obtained in \cite{mz13} have been extended to profinite groups in which all $w$-values are contained in a union of countably many closed subgroups \cite{demori}. Quite possibly, for profinite groups in which the commutators are covered by countably many procyclic subgroups some analogues of Theorem A and Theorem B hold true.

Though profinite groups constitute the main topic of the present study, the above results also have a bearing on the case of abstract groups. As we have 
already mentioned, the main result in \cite{FerShu} says that if $G$ is an abstract group whose commutators are covered by finitely many cyclic subgroups, then $G'$ is either finite or cyclic. Now we can deduce the following additional information.

\begin{thmC} \label{abstract} Let $G$ be a group that possesses $m$ cyclic subgroups whose union contains all commutators of $G$. Then $G$ has a characteristic subgroup $M$ contained in $G'$ such that the order of $M$ is $m$-bounded and $G'/M$ is cyclic.
\end{thmC}

Of course, the information provided by Theorem C is meaningful only in the case where $G'$ is finite, since otherwise $G'$ is cyclic.

\section{Finite groups with commutators covered by few cyclic subgroups}
We start with some elementary lemmas. 

\begin{lemma}
\label{reduction}
Let $n\geq1$ be a positive integer and let $H$ be a characteristic finite nilpotent subgroup of a group $G$. Assume that for every prime $p$ dividing the order of $H$ the Sylow $p$-subgroup $P$ of $H$  has a characteristic subgroup $M_p$, of order at most $n$, such that $P/M_p$ is cyclic. Then $G$ possesses a characteristic subgroup $M$ contained in $H$ such that the order of $M$ is $n$-bounded and $H/M$ is cyclic.
\end{lemma}

\begin{proof}
We take $M$ to be the product of all $M_p$, where $p$ ranges through the set of all prime divisors of the order of $H$. Obviously, $M_p=1$ whenever $p\geq n+1$ and therefore the order of $M$ is less than $n^n$. It is clear that, being the product of characteristic subgroups, $M$ is characteristic. The quotient $H/M$ is a nilpotent group with cyclic Sylow subgroups and therefore $H/M$ is cyclic. The proof is complete.
\end{proof}

\begin{lemma}
\label{normal-characteristic}
Let $H$ be a characteristic subgroup of an abstract (resp. profinite) group $G$. Suppose that $H$ possesses a normal finite subgroup $N$ such that $H/N$ is cyclic (resp. procyclic). Then $G$ has a  characteristic subgroup $M$ contained in $H$ such that the order of $M$ is at  most $|N|^{2}$ and $H/M$ is cyclic (resp. procyclic).
\end{lemma}

\begin{proof}
Take $M$ to be the subgroup generated by all elements of $H$ of order dividing $|N|$. It is clear that $M$ is a characteristic subgroup in $G$ containing $N$. Since the quotient $M/N$ is cyclic  (resp. procyclic) and generated by elements of order dividing $|N|$, it has order dividing $|N|$. It follows that the order of $M$ is at most $|N|^2$, as required.
\end{proof}

The following lemma is taken from \cite{AS1}. It will play an important role in our arguments.
\begin{lemma}
\label{pri}
Let $G$ be a finite noncyclic $p$-group that can be covered by $m$ cyclic subgroups. Then the order of $G$ is $m$-bounded.
\end{lemma}

Recall that in a group $G$ the subgroup $\gamma_\infty(G)$ is the intersection of all
$\gamma_i(G)$ for $i\in\N$. Clearly, a finite group $G$ is nilpotent if and only if $\gamma_\infty(G)=1$. It is an easy exercise to show that if $G$ is a finite group, then $\gamma_\infty(G)$ is generated by the commutators $[x,y]$ such that $x,y$ are elements of $G$ having mutually coprime orders. The following theorem was proved in \cite{AS1}.

\begin{thm}\label{glavk}
Let $G$ be a finite group that possesses $m$ cyclic subgroups whose union contains all commutators $[x,y]$ such that $x,y$ are elements of $G$ having mutually coprime orders. Then $\gamma_\infty(G)$ has a subgroup $\Delta$ such that
\begin{enumerate}
\item $\Delta$ is normal in $G$;
\item $|\Delta|$ is  $m$-bounded;
\item $\gamma_\infty(G)/\Delta$ is cyclic.
\end{enumerate}
\end{thm}

Further, we will require the following special case of a result of Guralnick
\cite[Theorem A]{Gur}.

\begin{thm}
\label{guralnick}
Let $G$ be a finite group in which $G'$ is an abelian $p$-group generated by at most two elements. Then every element of $G'$ is a commutator.
\end{thm}

We will now start our analysis of finite groups in which commutators are covered by at most $m$ cyclic subgroups. We recall that a finite group $G$ has rank $r$ if $r$ is the least integer such that every subgroup of $G$ can be generated by at most $r$ elements.

\begin{lemma}
\label{clatwo}
Let $G$ be a finite nilpotent group of class $2$ that possesses $m$ 
cyclic subgroups whose union contains all commutators of $G$.
Then $G'$ has a characteristic subgroup $M$ such that $|M|$ is $m$-bounded and $G'/M$ is cyclic.
\end{lemma}
\begin{proof} By Lemma \ref{reduction} it is sufficient to show that the claim is correct for each Sylow subgroup of $G$. Therefore we can assume that $G$ is a $p$-group for some prime $p$. Since $G$ is of class $2$, it follows that for each element $y\in G$ the subgroup $[G,y]$ consists entirely of commutators. By Lemma \ref{pri} there exists a bound $\beta$ such that either $[G,y]$ is cyclic or $|[G,y]|\leq\beta$. Let $M$ be the product of all 
subgroups of $G'$ whose order is at most $\beta$. Since $G'$ is an abelian group with at most $m$ generators, the rank of $G'$ is at most $m$. It follows that the order of $M$ is bounded as well. We pass to the quotient $G/M$ and 
we obtain that $[G,y]$ is cyclic for all $y\in G$.
Suppose that $G'$ is not cyclic. Passing to $G/\Phi(G')$, we assume that $G'$ is elementary abelian. We can choose $x,y\in G$ such that $[G,x]$ and $[G,y]$ are both nontrivial and $[G,x]\neq[G,y]$.
Now choose any element $t\in G$ which does not belong to $C_{G}(x)\cup C_{G}(y)$. Such an element $t$ exists because a group cannot be the union of two proper subgroups. Then $[G,t]=[G,x]$ because these are both cyclic groups of order $p$ containing $[x,t]\ne 1$. Similarly $[G,t]=[G,y]$, a contradiction.
\end{proof}
 
In what follows we write $O_{p'}(X)$ to denote the largest normal $p'$-subgroup of a finite group $X$.

\begin{lemma}\label{metab} Let $G$ be a finite metabelian group that possesses $m$ cyclic subgroups whose union contains all commutators of $G$. Then $G'$ has a characteristic subgroup $M$ such that $|M|$ is $m$-bounded and $G'/M$ is cyclic.

\end{lemma}
\begin{proof} By Lemma \ref{reduction} it is sufficient to show that each Sylow subgroup of $G'$ possesses a characteristic subgroup with the required properties. Let $p$ be a prime divisor of the order of $G'$. Passing to the quotient $G/O_{p'}(G')$ we can assume that $G'$ is a $p$-group. Let $a\in G'$ and $b\in G$. It is clear that each element of $\langle[a,b]\rangle$ has form $[a^i,b]$. If $x\in G$ we have $[xa^i,b]=[x,b]^{a^i}[a^i,b]=[x,b][a^i,b]$ and so every element in the coset $[x,b]\langle[a,b]\rangle$ is a commutator. Thus, the coset $[x,b]\langle[a,b]\rangle$ is covered by $m$ cyclic subgroups. It follows that for some $1\leq i\neq j\leq m+1$ one of the elements $[x,b][a^i,b]$ and $[x,b][a^j,b]$ is a power of the other. For simplicity, assume that 
$C_1=\langle [x,b][a^i,b]\rangle$ contains $[x,b][a^j,b]$.  Then also $[a^{j-i},b]\in C_1$. Therefore the subgroup 
$\langle C_1,[a,b]\rangle$ decomposes as a direct product $C_2\times D_2$, where $C_2$ is cyclic and $D_2$ is cyclic of order at most $m$.
Let $D$ be the product of all subgroups of $G'$ whose order is at most $m$. Since $G'$ is an abelian group with at most $m$ generators, the rank of $G'$ is at most $m$. It follows that the order of $D$ is at most $m^m$. We pass to the quotient $G/D$ and we see that
\begin{equation}\label{intermediate}
 \langle[a,b],[x,b]\rangle\textrm{ is cyclic for all }a\in G' \textrm{ and }b,x\in G.
\end{equation}
Note that (\ref{intermediate}) implies in particular that $[G',y]$ is cyclic for all $y\in G$.

Let us now show that $\gamma_3(G)$ is cyclic. We can pass to $G/\Phi(\gamma_3(G))$ and assume that $\gamma_3(G)$ is elementary abelian. 
Choose $y\in G$ such that $[G',y]\neq 1$ and $x$ outside $C_G(y)$. Since $[G',y]$ is of order $p$, by (\ref{intermediate}) we have $[G',y]\leq\langle[x,y]\rangle$. The same argument shows that $[G',x]\leq\langle[x,y]\rangle$. Using that $[G',x]$ is of order at most $p$, we conclude that $[G',x]\leq[G',y]$. This happens for all $x$ outside $C_G(y)$. Since the set of all such $x$ outside $C_G(y)$ generates the whole group $G$, it follows that $\gamma_3(G)=[G',y]$.

Thus, indeed $\gamma_3(G)$ is cyclic (we no longer assume that $\gamma_3(G)$ is elementary abelian). By Lemma \ref{clatwo}, $G'/\gamma_3(G)$  has a characteristic subgroup $K/\gamma_3(G)$ such that $|K/\gamma_3(G)|$ is $m$-bounded and $G'/K$ is cyclic. We see that $K$ is an abelian subgroup containing a cyclic subgroup of bounded index. The subgroup of $K$ generated by all elements of order at most the exponent of $K/\gamma_3(G)$ is characteristic and has bounded order. By factoring this subgroup we may assume that $K/\gamma_3(G)$ is cyclic and so $G'/\gamma_3(G)$ has rank at most two. Now Theorem \ref{guralnick} tells us that every element of $G'/\gamma_3(G)$ is a commutator. Therefore Lemma \ref{pri} shows that $G'/\gamma_3(G)$ either is cyclic or has $m$-bounded order.

Suppose that $G'/\gamma_3(G)$ is cyclic. Taking into account that $\gamma_3(G)$ is cyclic as well, we conclude that $G'$ is 2-generator. By Theorem \ref{guralnick} every element of $G'$ is a commutator. Therefore Lemma \ref{pri} shows that $G'$ either is cyclic or has $m$-bounded order, whence the lemma follows.

Suppose now that $G'/\gamma_3(G)$ has $m$-bounded order. We argue as above. Let $X$ be the product of all subgroups of $G'$ of order at most $|G'/\gamma_3(G)|$. As $G'$ has rank at most $3$, the subgroup $X$ has bounded order and $G'/X$ is cyclic. The proof is complete.
\end{proof}

\begin{thm}
\label{main}
Let $G$ be a finite group that possesses $m$ cyclic subgroups whose union contains all commutators of $G$. Then $G'$ has a characteristic subgroup $M$ such that the order of $M$ is $m$-bounded and $G'/M$ is cyclic.
\end{thm}

\begin{proof} Let $\Delta$ have the same meaning as in Theorem \ref{glavk}.
By Lemma \ref{normal-characteristic}, we may assume that $\Delta$ is characteristic in
$G$.
We can pass to the quotient $G/\Delta$ and suppose that $G$ is soluble with $\gamma_\infty(G)$ cyclic. The group $G$ acts on $\gamma_\infty(G)$ by conjugation and as the automorphism group of a cyclic group is abelian,
it follows that $G'$ centralizes $\gamma_\infty(G)$. Therefore $G'$ is nilpotent. By Lemma \ref{reduction} it is sufficient to show that each Sylow subgroup of $G'$ possesses a characteristic subgroup with the required properties. Let $p$ be a prime divisor of the order of $G'$. Passing to the quotient $G/O_{p'}(G')$ we can assume that $G'$ is a $p$-group.

Next we remark that since $G'$ is an $m$-generator $p$-group, the Burnside Basis Theorem \cite[III.3.15]{hup} shows that $G'$ is generated by $m$ commutators. Therefore we can choose at most 2$m$ elements in $G$ such that $G'$ is generated by commutators in the chosen elements. Without loss of generality we can assume that $G$ is generated by the chosen elements.

Let $x$ be a commutator. Then any conjugate of $x$ is again a commutator and so it belongs to at least one of the $m$ cyclic subgroups
covering the commutators of $G$. Since any finite cyclic subgroup has at most one subgroup of any given order, it follows that the subgroup $\langle x\rangle$ has at most $m$ conjugates. Thus, $N_G(\langle x\rangle)$ has index at most $m$. Set
$T=\cap \, N_G(\langle x\rangle)$, where $x$ ranges over all commutators in $G$. 
Since $G$ can be generated by $2m$ elements, it has only boundedly many subgroups of any given index \cite[Theorem 7.2.9]{mhall} and so $T$ has $m$-bounded index in $G$. 
Also, $T'\le C_G(x)$ for every commutator $x$, and so $T'$ centralizes $G'$. Therefore $T$ is metabelian and the derived length of $G$ is bounded. 

We will now use induction on the derived length of $G$. By induction, we can assume that $G''$ has a characteristic subgroup $M_1$ such that $|M_1|$ is $m$-bounded and $G''/M_1$ is cyclic. Passing to the quotient over $M_1$, we assume that $G''$ is cyclic. The group $G$ induces by conjugation an abelian group of automorphisms of $G''$. Hence, $G'$ centralizes $G''$ and thus the nilpotency class of $G'$ is at most $2$. 

By Lemma \ref{metab} (applied to $G/G''$) the derived group $G'$ has a characteristic subgroup $M$ containing $G''$, and such that $M/G''$ has $m$-bounded order while $G'/M$ is cyclic. As $|M:Z(M)|\le |M:G''|$, by the Schur Theorem \cite[Theorem 4.12]{Rob} $M'$ has $m$-bounded order as well. Factoring $M'$ out we can assume that $M$ is abelian. 
We can write $M=R\times M_2$, where $R$ is a cyclic group and $M_2$ is a subgroup of $m$-bounded order. 
It follows from Lemma \ref{normal-characteristic} that $M_2$ is contained in a characteristic subgroup of $G$ of $m$-bounded order. Factoring 
it out, we can assume that $M$ is cyclic. Moreover, $G$ acts on $M$ by conjugation so $[G',M]=1$. It follows that $G'/Z(G')$ is cyclic. We conclude that $G'$ is abelian and the theorem follows from Lemma \ref{metab}.
 \end{proof}

It is easy to see that under the hypothesis of Theorem \ref{main} the order of $G'$ cannot be bounded even if we know that $G'$ is noncyclic. Indeed, let $A$ be a finite group such that $A'$ is noncyclic of order four and let $B$ be a finite group such that $B'$ has odd prime order $p$. Set $G=A\times B$. Then $G'$ is noncyclic and covered by 3 cyclic subgroups. The order of $G'$ is 4$p$ and this tends to infinity when $p$ does so.

However, our next result shows that if $G$ is a $p$-group satisfying the hypothesis of Theorem \ref{main} and having noncyclic derived group $G'$, then the order of $G'$ is $m$-bounded.

\begin{thm}
\label{finite-p}
Let $p$ be a prime and let $G$ be a finite $p$-group in which all commutators can be covered by $m$ cyclic subgroups. Then either $G'$ is cyclic or the order of $G'$ is $m$-bounded.
\end{thm}

\begin{proof}
Let us assume that $G'$ is not cyclic. By Theorem \ref{main} the derived group $G'$ contains a characteristic subgroup $M$ of $m$-bounded order such that $G'/M$ is cyclic.

We choose in $G$ a normal subgroup $N$ of minimum possible order subject to the condition that $(G/N)'$ is cyclic. Then $1\ne N\subseteq G'$, and the order of $N$ is $m$-bounded. Since $G$ is a finite $p$-group, there exists a normal subgroup $L$ in $G$ which is contained in $N$ and has index $p$ in $N$. By the assumption on $N$, the derived group of $G/L$ is not cyclic.
Thus by factoring out $L$ we may assume that $N$ has order $p$ and therefore
$G'$ is 2-generator. Now Theorem \ref{guralnick} tells us that every element of $G'$ is a commutator. Hence, it follows from Lemma \ref{pri} that $G'$ has $m$-bounded order. The proof is complete.
\end{proof}

\section{Proofs of the main results}
\label{profinite}

We are now ready to complete the proofs of the theorems stated in the introduction.

\begin{proof}[Proof of Theorem A]
By Lemma \ref{normal-characteristic}, it suffices to find a normal subgroup of $G$ inside $G'$ with the desired properties.
Let $\mathcal N$ be the family of all open normal subgroups of $G$, and observe that
$G\cong \varprojlim_{N\in\mathcal{N}} \, G/N$.
Consider an arbitrary $N\in\mathcal{N}$, and put $Q=G/N$.
Let us write $\mathcal{M}(N)$ for the set of all subgroups $R$ of $Q'$ which are normal in
$Q$, of order at most $f(m)$, and satisfy the condition that $Q'/R$ is cyclic.
By Theorem \ref{main}, $\mathcal{M}(N)$ is not empty.

Given $L,N\in\mathcal{N}$ with $L\le N$, the natural map $\pi_{LN}$ from $G/L$ to $G/N$ induces a map $\varphi_{LN}$ from $\mathcal{M}(L)$ to $\mathcal{M}(N)$.
This way we get an inverse system $\{\mathcal{M}(N),\varphi_{LN},\mathcal{N}\}$ of finite sets.
By \cite[Proposition 1.1.4]{ribes-zal}, the corresponding inverse limit is not empty.
If $(M_N/N)_{N\in\mathcal{N}}$ is an element of that inverse limit, then
$\pi_{LN}(M_L/L)=M_N/N$ for all $L,N\in\mathcal{N}$ such that $L\le N$.
Hence we can form the inverse limit $\varprojlim_{N\in\mathcal{N}} \, M_N/N$, which corresponds to a normal closed subgroup $M$ of $G$.
Since $|M_N/N|\le f(m)$ for every $N\in\mathcal{N}$, we also have $|M|\le f(m)$.
Finally, observe that
\[
(G/M)' = G'/M \cong \varprojlim_{N\in\mathcal{N}} \,  \frac{(G/N)'}{M_N/N}
\]
is an inverse limit of cyclic subgroups, and so procyclic.
\end{proof}

\begin{proof}[Proof of Theorem B] 
Let $G$ be a pro-$p$ group such that all commutators in $G$ are covered by $m$ procyclic subgroups. Choose an open normal subgroup $N$ of $G$. The quotient $Q=G/N$ is a finite $p$-group satisfying the hypotheses of Theorem \ref{finite-p}. Therefore either $Q'$ is cyclic or the order of $Q'$ is at most some $m$-bounded number $k$. Suppose now that the derived group $G'$ is not of order at most $k$. Then there exists an open normal subgroup $N$ in $G$ such that the order of the derived group $(G/N)'$ is larger than $k$ and hence $(G/N)'$ is cyclic. Then $(G/H)'$ is cyclic for any open normal subgroup $H$ contained in $N$. It follows that $G'$ is procyclic, as required.
\end{proof}

\begin{proof}[Proof of Theorem C]  By the main result of \cite{FerShu} mentioned in the introduction, we know that $G'$ is either cyclic or finite. 
So it is sufficient to concentrate on the case where $G'$ is finite. There exists a finitely generated subgroup of $G$ whose derived subgroup coincides with $G'$, and consequently
we may assume that $G$ is finitely generated. As $G'$ is finite, the centralizer $C_G(G')$ has finite index in $G$ and so it is also finitely generated. Moreover, $C_G(G')$ is nilpotent of class at most $2$ and thus it is residually finite (see \cite{H}). We conclude that $G$ is residually finite as well. Since $G'$ is finite, there exists a normal subgroup $N$ of $G$ of finite index in $G$ such that $G'\cap N=1$. As $G'$ is isomorphic to $G'N/N=(G/N)'$ and in 
the finite group $G/N$ the result holds by Theorem \ref{main}, the conclusion follows.
\end{proof}

\end{document}